\documentclass[12pt,leqno]{amsart}%
\usepackage{times,amsmath,amsthm,amssymb,latexsym,amscd}
\usepackage[all]{xy}
\usepackage{graphicx}%
\usepackage{amsmath}%

\newtheorem{theorem}{Theorem}[section]
\newtheorem{thm}{Theorem}[section]
\newtheorem{lemma}[theorem]{Lemma}
\newtheorem{prop}[theorem]{Proposition}
\newtheorem{cor}[theorem]{Corollary}

\theoremstyle{definition}

\newtheorem{defn}[theorem]{Definition}

\newtheorem{examples}[theorem]{Examples}

\newtheorem{nota}[theorem]{Notation}

\theoremstyle{remark}
\newtheorem{remark}[theorem]{Remark}
\newtheorem{remarks}[theorem]{Remarks}
\numberwithin{equation}{section}

\def\C{{\mathbb{C}}}
\def\R{{\mathbb{R}}}

\def\Z{{\mathbb{Z}}}

\def\bs{{\backslash}}

\def\E{{\mathcal{E}}}

\def\nab{{\nabla}}

\def\E{{\mathcal{E}}}
\def\Lk{{L^{\otimes k}}}

\def\H{{\mathcal{H}}}
\def\Id{\operatorname{Id}}
\def\Spec{\operatorname{\textit{Spec}}}
\def\trace{\operatorname{\textit{trace}}}
\def\vol{\operatorname{\textit{vol}}}
\def\Sp{\operatorname{\textit{Sp}}}

\def\bs{{\backslash}}

\def\E{{\mathcal{E}}}
\def\L{{\mathcal{L}}}

\def\nab{{\nabla}}

\def\E{{\mathcal{E}}}
\def\Lk{{L^{\otimes k}}}

\def\H{{\mathcal{H}}}
\def\rtm{{\R^{2m}}}
\def\ztm{{\Z^{2m}}}

\def\n{{\mathfrak{n}}}
\def\z{{\mathfrak{z}}}
\def\tx{{\tilde{X}}}
\def\tip{{2\pi i}}
\def\cik{{C^\infty_k(P,\C)}}
\def\rb{{\mathbf{r}}}
\def\a{\mathbf a}
\def\b{\mathbf b}

\def\j{{\mathbf{j}}}
\def\Ad{\operatorname{Ad}}

\begin{document}\begin{title}[Classical Equivalence and
Quantum Equivalence for Magnetic Fields]
{Classical Equivalence and Quantum Equivalence of Magnetic Fields on Flat Tori}
\end{title}

\author{Carolyn Gordon}
\address{Department of Mathematics, Dartmouth College, Hanover, NH  03755}
\email{csgordon@dartmouth.edu}

\author{William Kirwin}
\address{CAMGSD, Departamento de Matem\'atica, Instituto Superior T\'ecnico, Av.
Rovisco Pais, 1049-001 Lisboa, Portugal \newline *Current Address:
Mathematisches Institut, University of Cologne, Weyertal 86 - 90, 50931
Cologne, Germany}
\email{will.kirwin@gmail.com}

\author{Dorothee Schueth}
\address{Institut f\"ur Mathematik, Humboldt-Universit\"at zu
Berlin, D-10099 Berlin, Germany}
\email{schueth@math.hu-berlin.de}

\author{David Webb}
\address{Department of Mathematics, Dartmouth College, Hanover, NH  03755}
\email{david.l.webb@dartmouth.edu}

\thanks{The first author and last author were supported in part by NSF Grants
DMS 0605247 and DMS 0906169.   The third author was partially supported by
DFG Sonderforschungsbereich 647. Moreover, the third author thanks Dartmouth
College and its Harris visiting program for great hospitality and support.}


\begin{abstract}
Let $M$ be a real $2m$-torus equipped with a translation-invariant metric $h$
and a translation-invariant symplectic form $\omega$; the latter we interpret as
a magnetic field on $M$. The Hamiltonian flow of half the norm-squared function
induced by $h$ on $T^*M$ (the ``kinetic energy'') with respect to the twisted
symplectic form $\omega_{T^*M}+\pi^*\omega$ describes the trajectories of a
particle moving on $M$ under the influence of the magnetic field $\omega$. If
$[\omega]$ is an integral cohomology class, then we can study the geometric
quantization of the symplectic manifold
$(T^*M,\omega_{T^*M}+\pi^*\omega)$ with the kinetic energy Hamiltonian.  We say
that the quantizations of two such
tori $(M_1,h_1,\omega_1)$ and $(M_2,h_2,\omega_2)$ are \emph{quantum equivalent}
if their quantum spectra, i.e., the
spectra of the associated quantum Hamiltonian operators, coincide; these quantum
Hamiltonian operators are proportional to the $h_j$-induced bundle Laplacians on
powers of the Hermitian line bundle on $M$ with Chern class $[\omega]$.

In this paper, we construct continuous families $\{(M,h_t)\}_t$ of mutually
nonisospectral flat tori $(M,h_t)$, each endowed with a translation-invariant
symplectic structure~$\omega$, such that the associated classical Hamiltonian
systems are pairwise equivalent. If $\omega$ represents an integer cohomology
class, then the $(M,h_t,\omega)$ also have the
same quantum spectra. We show moreover that
for any translation-invariant metric~$h$ and any translation-invariant symplectic structure~$\omega$ on~$M$ that represents
an integer cohomology class, the associated
quantum spectrum determines whether $(M,h,\omega)$ is K\"ahler,
and that all translation-invariant K\"ahler structures $(h,\omega)$
of given volume on~$M$ have the same quantum spectra.
Finally, we construct pairs of magnetic fields
$(M,h,\omega_1)$, $(M,h,\omega_2)$
having the same quantum spectra but nonsymplectomorphic classical
phase spaces. In some of these examples the pairs consist of K\"ahler
manifolds.
\end{abstract}

\maketitle

\section{Introduction}

\noindent
Consider an even-dimensional torus
$M=\ztm\bs\rtm$.  To each translation-invariant closed 2-form $\omega$ and
translation-invariant (i.e., flat) Riemannian metric $h$ on $M$, associate a
Hamiltonian system $(T^*M, \Omega, H)$.  Here $\Omega$ is the symplectic form on
$T^*M$ given by $\Omega=\omega_0+\pi^*\omega$, where $\omega_{0}$ is the
Liouville form, and $\pi:T^{\ast}M\rightarrow M$ is the projection.  The
Hamiltonian
function $H$ is given by $H(q,\xi)=\frac{1}{2}h_q(\xi,\xi)$.  In case
$\omega=0$, the Hamiltonian system gives the classical geodesic flow.  A
nontrivial closed
$2$-form $\omega$ may be viewed as a magnetic field on~$M$, and the Hamiltonian
system describes the dynamics of a charged particle moving in the magnetic
field.  We will say that $(M,h_1,\omega_1)$ and
$(M,h_2,\omega_2)$ are \emph{classically equivalent} if the associated
Hamiltonian systems are equivalent, i.e., if there is a symplectomorphism of
cotangent bundles intertwining the Hamiltonian functions.

If, moreover, $\omega$ represents an integer cohomology class, then there exists
a Hermitian complex line bundle $L$ with Chern class $[\omega]$. Choose a
Hermitian connection $\nabla$ with curvature $-2\pi i\omega$.  The connection
gives rise to a Hermitian connection, also denoted $\nabla$, on each tensor
power $\Lk$, i.e., on the line bundles
with Chern class $k\omega$, $k\in\Z^+$.

According to the procedure of geometric quantization (specifically, with respect
to the vertical polarization on the cotangent bundle in the presence of the
metaplectic correction), the quantum Hilbert space at level $\hbar=1/k$
($k\in\Z^+$) associated to $(T^*M,\omega_0+\pi^*\omega)$ is the $L^2$-space of
square integrable sections of $\Lk$. The quantum Hamiltonian associated to the
classical Hamiltonian $H$ is the operator
$\widehat{H}_{k}=\frac{\hbar^{2}}{2}\Delta$, where $\Delta=-\trace(\nab^2)$.
(See \cite{Woo}, and note that the scalar curvature
term appearing there is zero in our case.  Also see Section~2 of~\cite{GKSW} for
a brief outline of geometric quantization.)

For technical reasons, we will always assume that $\omega$ is nondegenerate,
i.e., that it is a symplectic structure on $M$. Of course, there are more
general magnetic fields on $M$, described by degenerate $2$-forms, but
nondegeneracy is crucial for certain isospectrality results (\textit{c.f.}
Remark \ref{rmk:nondeg}). We will see in Lemma~\ref{indep} that the spectra of
the operators $\widehat{H}_{k}$ are independent of the choice of the connection
$\nabla$ with curvature $-2\pi i\omega$.   Hence the spectra depend only on
$\omega$, $h$, and of course $k$, and will be denoted by $Spec(k\omega, h)$.
(This independence of the choice of connection is special to our setting of flat
tori with translation-invariant nondegenerate $\omega$.)   We will say that
$(M,h_1,\omega_1)$ and $(M,h_2,\omega_2)$ are \emph{quantum equivalent} if
$Spec(k\omega_1, h_1)=Spec(k\omega_2, h_2)$ for all $k\in\Z^+$.

Our main results are:

\begin{thm}\label{1}  Let $\omega$ be any translation-invariant symplectic
structure on $M:=\ztm\bs\rtm$.  Then every translation-invariant metric $h$ on
$M$ lies in a continuous family $\{h_t\}$ of mutually nonisometric flat metrics
such that $(M,h_t,\omega)$ is classically equivalent to $(M,h,\omega)$ for all
$t$.  Moreover, if $\omega$ represents an integer cohomology class, then these
$(M,h_t,\omega)$ are also quantum equivalent to  $(M,h,\omega)$ for all $t$.
\end{thm}

For the remainder of the results, we assume that the forms $\omega_i$ ($i=1,2$)
represent integer cohomology classes.   In Theorem~\ref{main}, we give necessary
and sufficient conditions for quantum equivalence of pairs $(M,h_1,\omega_1)$
and $(M,h_2,\omega_2)$, and we observe that in our setting, for any choice of
$\omega$ as above, $Spec(\omega,h)$ determines $Spec(k\omega, h)$ for all
$k\in\Z^+$.

We will say that $(M,h,\omega)$ is \emph{K\"ahler}, or that $(h,\omega)$ is a
\emph{K\"ahler structure} on $M$, if there exists a complex structure $J$ such
that $(M, h,\omega,J)$ is K\"ahler.

We then prove the following, for $M=\ztm\bs\rtm$ with $m$ arbitrary:

\begin{thm} For any translation-invariant symplectic form~$\omega$
and translation-invariant metric~$h$ on~$M$, the spectrum
$Spec(\omega,h)$ determines whether $(M,h,\omega)$ is K\"ahler.
Moreover, all translation-invariant K\"ahler structures $(h,\omega)$ of given
volume on $M$ are quantum equivalent.  (Here both $\omega$ \textit{and} $h$ are
allowed to vary.)
\end{thm}

\begin{thm}\label{2}  The collection $Spec(k\omega,h)$, $k\in \Z^+$, does not
determine the symplectic structure $\omega$ on $M$ nor the symplectic structure
$\Omega=\omega_0+\pi^*\omega$ on $T^*M$ (nor the restriction of $\Omega$ to the
cotangent bundle with the zero section removed).  In particular, quantum
equivalent systems need not have the same classical phase space.
\end{thm}

We pause to clarify the notion of classical phase space used here and to
motivate the parenthetical remark in Theorem~\ref{2}.   By considering the
entire cotangent bundle, instead of the cotangent bundle minus its zero section,
we are using a somewhat stronger notion of equivalence than is sometimes
considered in the mathematical literature. Indeed, our notion of classical
equivalence (Definition \ref{def:classical-equiv}) implies that if
$(M_1,h_1,\omega_1)$ and $(M_2,h_2,\omega_2)$ are classically equivalent, then
$(M_1,\omega_1)$ and $(M_2,\omega_2)$ are symplectomorphic.  The removal of the
zero section is mathematically rather than physically motivated.
Often analytical considerations necessitate replacing the Hamiltonian flow by a
reparametrization that is not well behaved on the zero section.  This is the
case, for example, in the analysis of the singularities of the wave trace
\cite{DG} and in the study of regularizations of the Kepler flow \cite{HdL},
\cite{LS}.  Removing the zero section also results in stronger --- and more
difficult --- geodesic rigidity results, as in the article \cite{CK} cited
below.
On the other hand, in classical mechanics, the phase space is the space of all
possible states of the system. For a particle moving on a manifold under
the influence of a magnetic field, an initial condition consisting of a given
position and zero momentum (i.e., an element of the zero section of $T^*M$) is
perfectly acceptable.
While the results above were stated using the phase space $(T^*M, \Omega)$, they
remain true if one removes the zero section from $T^*M$. In particular, the
resulting stronger version of Theorem~\ref{2} (the parenthetical comment) is
proven in Proposition~\ref{noneq0}.

Theorem~\ref{1} contrasts sharply with the case $\omega=0$.   C. Croke and B.
Kleiner \cite{CK} showed that the geodesic flow on a torus is $C^0$ rigid, i.e.,
that any Riemannian manifold whose geodesic flow is $C^0$ conjugate to that of a
flat torus $(M,h)$ is isometric to the torus $(M,h)$.  Note that $C^0$-conjugacy
is a much weaker condition than classical equivalence.

This is the second of two articles addressing questions of quantum equivalence.
In the first \cite{GKSW}, we constructed examples of pairs (or finite families)
of Hermitian locally symmetric spaces $M_i$ for which the line bundles with
Chern class defined by the K\"ahler structure and their tensor powers over the
various $M_i$
are isospectral for all $i$.

This article was motivated by results of \cite{GGKW}.   In fact, our results on
quantum equivalence of  magnetic fields are a reinterpretation and expansion of
Corollaries 3.8 and 3.9 of \cite{GGKW}.

\section{Classical equivalence of magnetic flows}

\begin{defn}\label{def:classical-equiv} Given a Riemannian manifold $(M,h)$ and
a closed 2-form $\omega$ on
$M$ (which we will always assume to be nondegenerate), let $\Omega$ be the
symplectic structure on $T^*M$ given by
$\Omega:=\omega_{0}+\pi^{\ast}\omega$, where $\omega_{0}$ is the Liouville form
(i.e., $\omega_0=-d\lambda$, where $\lambda$ is the canonical $1$-form on
$T^*M$) and $\pi:T^{\ast}M\rightarrow M$ is the projection.  Define
$H:T^*M\to\R$ by $H(q,\xi)=\frac{1}{2}h_q(\xi,\xi).$   We will refer to
$(T^{\ast}M, \Omega, H)$ as the classical Hamiltonian system associated with
$(M,h,\omega)$.   Given Riemannian manifolds $(M_i,h_i)$, $i=1,2$, and closed
2-forms $\omega_i$ on $M_i$, we will say that $(M_1,h_1,\omega_1)$ and $(M_2,
h_2,\omega_2)$ are
\emph{classically equivalent} if the associated Hamiltonian systems
$(T^{\ast}M_i, \Omega_i, H_i)$ are equivalent, i.e., if there exists a
symplectomorphism $\Phi:(T^{\ast}M_1, \Omega_1)\to (T^{\ast}M_2, \Omega_2)$ such
that $H_1=H_2\circ \Phi$.
\end{defn}

\begin{theorem}\label{class} Let $\omega$ be a translation-invariant symplectic
structure on $\rtm$, let $A$ be a linear symplectomorphism of $(\rtm,\omega)$,
let $h$ be a translation-invariant metric on $\rtm$, and let $\L$ be a lattice
in $\rtm$.  We will continue to denote by $\omega$ and $h$ the induced
structures on quotients of $\R^{2m}$ by a lattice.    Then
$(\L\bs\rtm,h,\omega)$ is classically equivalent to $(A(\L)\bs\rtm,h,\omega)$.
\end{theorem}

\begin{remarks}\label{r}
$\text{ }$

\noindent (i) The conclusion may be rephrased as the statement that $(\L\bs\rtm,
A^*h, \omega)$ is classically equivalent to $(\L\bs\rtm, h,\omega)$.

\noindent (ii) In Theorem~\ref{class}, we do not require that $\L$ have maximal
rank in $\rtm$, i.e., that $\L\bs\rtm$ be a torus.   However, in the case that
it is a torus and that $\omega$ represents an integer cohomology class in
$\L\bs\rtm$, the reformulation in (i) will give us different quantum
Hamiltonians (Laplacians associated with different metrics) on the same complex
line bundle.   We will see in Corollary~\ref{contin} that the systems are
quantum equivalent.

\end{remarks}

\begin{proof} Let $n=2m$.  Under the standard identification of $T^*\R^n$ with
$\R^{2n}$, the symplectic form $\Omega=\omega_0+\pi^*\omega$ is a translation-invariant 2-form and thus may be identified with the bilinear form on $\R^{2n}$
with matrix
$$\begin{bmatrix}C&\Id\\-\Id&0\end{bmatrix}
$$
with respect to the standard basis,
where each block is of size $n\times n$ and where $C$ is the matrix of the
anti-symmetric nondegenerate bilinear form on $\R^n$ defined by $\omega$.  The
linear map $\Phi:\R^{2n}\to\R^{2n}$ given by
$$\Phi(q,p)=(Aq+C^{-1}({}^tA^{-1}-\Id)p, p)
$$
preserves $\Omega$, as can be seen by an easy computation
using ${}^t C=-C$ and ${}^t\!A C A =C$.   The Hamiltonian $H$ depends only on
$p$ (since $h$ is translation invariant) and thus is also preserved by $\Phi$.
Thus $\Phi$ is a self-equivalence of the Hamiltonian system $(T^*\R^n,h,
\omega$).   Finally,
we have
$\Phi(q_0+q,p)=(Aq_0,0)+\Phi(q,p)$ for all $q_0\in\R^n$ and, in particular, for
all $q_0\in\L$.  Thus $\Phi$ induces an equivalence between
$(\L\bs\rtm,h,\omega)$ and $(A(\L)\bs\rtm,h,\omega)$.
\end{proof}

\begin{cor}\label{corclass}  Let $\omega$ be a translation-invariant symplectic
structure on a torus $M=\ztm\backslash\rtm$.  Then every translation-invariant
Riemannian metric $h$ on $M$ belongs to a continuous family $\{h_t\}_t$ of
mutually nonisometric translation-invariant Riemannian metrics such that
$(M,h_t,\omega)$ is classically equivalent to $(M,h,\omega)$ for all $t$.  The
parameter space of this deformation has dimension at least $2m$.
\end{cor}

\begin{proof} $h_t$ is defined as
$A_t^*h$, where $A_t$ (with $A_0=\Id$) is a curve in the
group of linear isomorphisms of $\rtm$ that preserve $\omega$.  This group is
isomorphic to $\Sp(2m,\R)$
and has dimension $m(2m+1)$, while the group of linear isomorphisms that
preserve $h$ is isomorphic to $O(2m)$ and has dimension $m(2m-1)$.  The
corollary thus follows from Remark~\ref{r}(i).
\end{proof}

\section{Hermitian line bundles over tori}

\subsection{Hermitian connections with the same translation-invariant curvature}
\label{same}
Let $M$ be a compact $C^\infty$ manifold, $L$ a Hermitian line bundle over $M$,
and let $[\omega]\in H^2(M;\Z)$ be the Chern class of $L$.  The curvature of any
Hermitian connection on $L$ lies in $-\tip[\omega]$. (Our notation differs from
that of \cite{GKSW} by a factor of $2\pi$.) If $\nabla$ and $\nabla'$ are two
Hermitian connections on $L$, then $\nabla'=\nabla+\tip \beta$ for some
real-valued $1$-form $\beta$ on $M$. The two connections have the same curvature
if and only if $d\beta=0$, in which case $\beta=\alpha + df$ for some harmonic
$1$-form $\alpha$ and some $f\in C^\infty(M)$. The term $df$ changes the
connection only by a gauge equivalence: in fact, letting $\E(L)$ denote the
space of smooth sections of $L$, then the map $\E(L)\to \E(L)$ given by
$s\mapsto e^{\tip f}s$ intertwines $\nabla+\tip df$ and $\nabla$.  Given any
Riemannian metric $h$ on $M$, this map also intertwines the Laplacians
$-\trace(\nabla +\tip df)^2$ and $-\trace(\nabla^2)$. Thus the two Laplacians
are isospectral.  The same statement holds for the associated Laplacians on all
the higher tensor powers of $L$. Thus we may assume that $f=0$.

In general, the addition of a harmonic 1-form $\tip \alpha$ to $\nabla$ will
affect the spectrum. However, we will see that in the case of line bundles with
nondegenerate Chern class over flat tori, endowed with a connection whose
curvature form on the torus is translation invariant,
the addition of a harmonic term does \emph{not} affect the spectrum; see
Lemma~\ref{indep} below.  Thus in this case, the spectrum of the
Laplacian depends only on the metric on the torus and the curvature of the
connection on the bundle.

\subsection{Principal circle bundles over tori.} Let $M=\Z^{2m}\bs \R^{2m}$,
where $m$ is a positive integer.  Let $\omega$ be a translation-invariant
symplectic structure on $M$ that represents an integer cohomology class.   We
will first construct a principal circle bundle~$P$ with Chern class $[\omega]$.
The bundle~$P$ will be a quotient by a discrete subgroup of a two-step nilpotent
Lie group~$N$, isomorphic to the Heisenberg group of dimension $2m+1$.

Since $\omega$ is translation invariant, it may be viewed as a nondegenerate
antisymmetric bilinear map $\omega:\R^{2m}\times\R^{2m}\to\R$ that takes integer
values on $\Z^{2m}\times\Z^{2m}$. We endow $N:=\R^{2m+1}$ with the structure of
a 2-step nilpotent Lie group with multiplication
$$
(u_1,t_1)(u_2,t_2)=(u_1+u_2, t_1+t_2+\tfrac12\omega(u_1,u_2))
$$
for all $u_1, u_2\in\rtm$ and $t_1, t_2\in\R$. Then $N$ is isomorphic to the
$(2m+1)$-dimensional Heisenberg group.  The coordinate vector field
$Z:=\frac{\partial}{\partial t}$ is left invariant and spans the center
$\z=\{0\}\times\R$ of the Lie algebra~$\n$ of~$N$. The center coincides
with the derived algebra, so the Lie bracket may be viewed as a bilinear map
$[\,\,,\,]:\rtm\times\rtm\to\z$, which is given by
$$[X,Y]=\omega(X,Y)Z.$$

Let $\Gamma\subset N$ be the subgroup generated by
$(e_1,0),\ldots,(e_{2m},0),(0,1)\in\R^{2m+1}=N$, where $e_j$ denotes the $j$th
standard basis vector of~$\rtm$.  Then the projection of~$\Gamma$ to~$\rtm$ is
$\ztm$. The intersection of~$\Gamma$ with the center $\{0\}\times\R$ of~$N$ is
precisely $\{0\}\times\Z$, the subgroup generated by the element $(0,1)$. In
fact, for $X,Y\in\{\pm e_1,\ldots,\pm e_{2m}\}$ the commutator
$(X,0)(Y,0)(X,0)^{-1}(Y,0)^{-1}$ equals
$(X+Y,\frac12\omega(X,Y))(-X-Y,\frac12\omega(-X,-Y))=(0,\omega(X,Y))$ which lies
in $\{0\}\times\Z$ since $\omega$ is integer valued on $\Z^{2m}\times\Z^{2m}$;
moreover, any product $(X_1,0)\cdot\ldots\cdot(X_k,0)$ with $X_1,\ldots,
X_k\in\{\pm e_1,\ldots,\pm e_{2m}\}$ and $X_1+\ldots+ X_k=0$ can be written as a
product of commutators as above.

In particular, $\Gamma$ is a uniform discrete subgroup of~$N$.  Set $P=\Gamma\bs
N$.  The center of~$N$ projects to a circle, and the action of the center by
translations on~$N$ induces a circle action on~$P$, giving $P$ the structure of
a principal circle bundle over~$M$.

We identify the circle $S^1$, given by the quotient of the center of~$N$ by its
intersection with~$\Gamma$, with the unitary group $U(1)$.  Its Lie algebra is
thus identified with the space of purely imaginary complex numbers. Under this
identification, the vector $Z\in\z$ above corresponds to $\tip\in i\R=T_1U(1)$;
hence, a connection on~$P$ is specified by an $S^1$-invariant 1-form $2\pi i\mu$
on $P$ such that
$2\pi i\mu(Z)\equiv 2\pi i$; that is, $\mu(Z)\equiv 1$.  (Here $\mu$ is
real-valued.) The kernel $\H$ of $\mu$ is called the horizontal distribution
associated with the connection.  By abuse of
terminology, we will say that $\mu$ is \emph{left invariant} if it pulls back to
a left-invariant $1$-form on $N$. In this case, $\H$ is spanned by
left-invariant vector fields (again in the sense that a left-invariant vector
field on $N$ induces a well-defined vector field on $P=\Gamma\bs N$, which we
refer to as left invariant) and thus may be viewed as a subspace of $\n$
complementary to $\z$.   Conversely, since every left-invariant 1-form is also
$S^1$ invariant, any complement of $\z$ in $\n$ is the horizontal distribution
associated with some translation-invariant connection on $P$.

Suppose that $\tip\mu$ is a left-invariant connection on $P$.  For $X,Y\in\H$,
we have
$$\tip d\mu(X,Y)=-\tip\mu([X,Y])=-\tip\mu(\omega(X,Y)Z)=-\tip\omega(X,Y)$$
since $\mu(Z)=1$.  Thus every translation-invariant connection on $P$ has
curvature form $-\tip\omega$.

Let $\alpha:\rtm\to\R$ be a linear functional.   Because of the nondegeneracy of
$\omega$, the map $\n\to\n$ that sends $X\in\rtm$ to $\alpha(X)Z$ and sends $Z$
to zero is an inner derivation of $\n$, and the map $N\to N$ given by
$(u,t)\mapsto (u,t+\alpha(u))$ is an inner automorphism of $N$.

It follows that if $\mu'$ is another left-invariant 1-form such that
$\mu'(Z)=1$, then $\mu'=\mu\circ \Ad(a)$ for some $a\in N$, and the
corresponding horizontal distribution satisfies $\H'=\Ad(a^{-1})\H$.

\subsection{Associated Hermitian line bundles}\label{assochlb}
Let $\omega$ and $P$ be as above and let $\tip\mu$ be a left-invariant
connection on $P$. The group $S^1=U(1)$ acts on $\C$ in the standard way, hence
diagonally on the product $P\times\C$, giving rise to a Hermitian line bundle
$$L=(P\times\C)/\sim$$
where $\sim$ is the equivalence relation given by $(p,w)\sim(pz^{-1},zw)$ for
$p\in P$, $w\in \C$, and $z\in S^1=U(1)$. The bundle $L$ has Chern class
$[\omega]$.

The higher tensor powers of $L$ are given by
$$\Lk=(P\times\C)/\sim_k$$
where $\sim_k$ is given by $(p,w)\sim_k(pz^{-1},z^kw)$
for $p\in P$, $w\in \C$, and $z\in S^1=U(1)$.

The space $C^\infty(M,\Lk)$ of smooth sections of $\Lk$ may be identified with
the subspace $C^\infty_k(P,\C)$ given by
\begin{equation}\label{eq.sect}
\cik=\{f\in C^\infty(P,\C)\mid f(pz^{-1})=z^kf(p)\text{ for all } p\in P, z\in
S^1=U(1) \}.
\end{equation}
Equivalently,
\begin{equation}\label{eq.sect2}
\cik=\{f\in C^\infty(P,\C)\mid Zf=-\tip kf\}.
\end{equation}

Because of the trivialization $TP\cong P\times\n$, any complex $1$-form on~$P$
may be viewed as a map from~$\n$ to the space of smooth complex functions
on~$P$.  For $f\in\cik$, the map corresponding to the $1$-form
$df+\tip kf\mu$ actually maps~$\n$ to $\cik$ and vanishes on $\z$; hence, it
induces a well-defined map from $\rtm$ to $\cik$. Recalling
Equation~(\ref{eq.sect}) and identifying $\rtm$ with the tangent space at each
point of $M$, we thus get a map $\nabla f: TM\to C^\infty(M,\Lk)$.  This defines
the Hermitian connection $\nabla$ on $\Lk$ associated with the connection
$\tip\mu$ on the principal bundle.  (Here we are
using the same notation $\nabla$ for the connection on each of the
bundles~$\Lk$.  The connection $\nabla$ on $\Lk$ is of course the usual
connection on the $k$th tensor power of the bundle $L$ arising from the
connection $\nabla$ on $L$.)  For $X\in TM$ and $\tx$ any horizontal vector in
$TP$ with $\pi_*\tx=X$, where $\pi:P\to M$ is the bundle projection, we have
$$\nabla_X\,f=\tx f.$$  The curvature of $\nabla$ is $-\tip k\omega$.

Given a flat metric $h$ on $M$ (i.e., an inner product on $\rtm$), let
$\{X_1,\dots, X_{2m}\}$ be an orthonormal basis of the Lie algebra $\rtm$ of
$M$, and let $\tx_1,\dots, \tx_{2m}$ be the horizontal lifts to vector fields on
the principal bundle $P$.  Then under the identification of $C^\infty(M,\Lk)$
with $\cik$ as in Equation~(\ref{eq.sect}), the Laplacian on $C^\infty(M,\Lk)$
defined by the connection $\nabla$ is given by
$$\Delta(f)=-\sum_{j=1}^{2m}\,\tx_j^2(f).$$

Let $\rho$ denote the representation of the nilpotent Lie group $N$ on $L^2(P)$
given by $(\rho(a)f)(p)=f(pa)$ and let $\rho_*$ be the representation of the Lie
algebra $\n$ given by the differential of $\rho$.  Then by Fourier decomposition
with respect to the action of the center of $N$, we have
$$L^2(P)=\oplus_{k\in\Z}\,L^2_k(P)$$
where
$$L^2_k(P)=\{f\in L^2(P)\mid \rho(z^{-1})f=z^k f\text{ for all }z\in
S^1=U(1)\}.$$
I.e., $L^2_k(P)$ is the closure of $\cik$ in $L^2(P)$.  Given a translation-invariant connection on $L$ (and thus on $\Lk$ for all $k\in\Z^+$) and a flat
metric on $M$, the associated Laplacian, viewed as an operator on $\cik$,
extends to $L^2_k(P)$ as the densely defined operator
\begin{equation}\label{eq.lap}\Delta=-\sum_{j=1}^{2m}\,\rho_*(\tx_j)^2.
\end{equation}

\begin{lemma}\label{indep}
We continue to assume that $\omega$ is a translation-invariant symplectic
structure on the torus $M$ and that the cohomology class of~$\omega$ is
integral.  Let $L$ be a Hermitian line bundle with Chern class $[\omega]$, and
let $\nabla$ and $\nabla'$ be two connections on $L$ with curvature
$-\tip\omega$. Then given any flat metric on~$M$, the Laplacians, and thus the
quantum Hamiltonians, on $\Lk$ defined by $\nabla$ and $\nabla'$ are isospectral
for all $k\in\Z^+$.
\end{lemma}

\begin{proof}
Note that $L$ is determined by its Chern class up to a bundle isomorphism
inducing the identity map on~$M$. Such isomorphisms preserve
the curvature forms of the connections which they intertwine, and the spectra of
the corresponding bundle Laplacians coincide.  Therefore, we may assume that $L$
is the Hermitian line bundle which we explicitly constructed above.

Let $\nabla$ be the connection associated with the principal connection
$\tip\mu$ as above. By the discussion in Subsection~\ref{same}, we
may assume that $\nabla'=\nabla +\tip\alpha$ for some harmonic 1-form $\alpha$
on~$M$.  Viewing $\alpha$ as a linear functional on $\rtm$, the map $\n\to\n$
given by $X+cZ\mapsto \alpha(X)Z$ (for all $X\in \rtm$ and $c\in\R$) is an inner
derivation and $\nabla'$ is the connection on~$L$ associated with a principal
connection $\tip\mu\circ \Ad(a)$ for some $a\in N$.  The horizontal distribution
$\H'$ is given by $\Ad(a^{-1})\H$.   It follows that the Laplacian associated
with $\nabla'$ on $\cik$ is given by
$$
\Delta'=\sum_{j=1}^{2m}\,
\rho_*(\Ad(a^{-1})\tx_j)^2=\sum_{j=1}^{2m}\rho(a^{-1})\rho_*(\tx_j)^2\rho(a)=\rho
(a^{-1})\circ \Delta\circ\rho(a).
$$
\end{proof}

\begin{remark}
\label{rmk:nondeg}
The hypothesis of nondegeneracy of $\omega$ is essential here.
At the other extreme in which $\omega=0$ so that $L$ is the trivial bundle, the
spectra of the various Laplacians $-(d-\tip\alpha)^2$ associated with the
(harmonic) connections of curvature zero form the Bloch spectrum of the torus.
\end{remark}

\begin{nota}\label{notspec}
In the notation of Lemma~\ref{indep}, we will write
$$\Spec(k\omega,h)
$$
for the spectrum of the operator $\widehat{H}_{k}=\frac{\hbar^{2}}{2}\Delta$,
where $\hbar=\frac{1}{k}$ and $\Delta$ is the Laplacian on~$\Lk$ defined by the
flat metric~$h$ on $\ztm\bs\rtm$ and any connection~$\nabla$ on~$L$ with
curvature $-\tip\omega$.  By the lemma, this spectrum is well defined.
\end{nota}

\section{Quantum equivalent line bundles}

\begin{nota}
Denote the standard coordinates on $\rtm$ by
$(x,y)=(x_1,\dots,x_m,y_1,\dots,y_m)$.  Given an $m$-tuple $\rb=(r_1,\dots,r_m)$
of positive integers such that
\begin{equation}
\label{div}r_1\mid r_2\mid\ldots\mid r_m,
\end{equation}
define a translation-invariant symplectic form $\omega_\rb$ on $\rtm$ by
$$\omega_\rb=\sum_{j=1}^m \,r_j\,dx_j\wedge dy_j.$$
\end{nota}

\begin{prop}\label{comp}\cite[p. 304]{GrHa}
Let $\omega$ be a translation-invariant symplectic structure on $\rtm$ such that
$[\omega]\in H^2(M;\Z)$.  Then there exists a unique $m$-tuple $\rb$ satisfying
Equation~\ref{div} such that $A^*\omega=\omega_\rb$ for some $A\in SL(2m,\Z)$.
We refer to the entries of this $m$-tuple as the \emph{Chern invariant factors}.
\end{prop}

Thus by a linear change of coordinates preserving $\ztm$, we may assume when
convenient that $\omega=\omega_\rb$ for some $\rb$ satisfying~\ref{div}.

\begin{remark}  Line bundles are of course classified by their Chern classes,
not by the Chern invariant factors.  The $m$-tuple $\rb$ is a complete
homeomorphism invariant of $\omega$, in the sense that given two integral
symplectic structures with the same Chern invariant factors, there is a
self-homeomorphism of the base space pulling back one integral symplectic
structure to the other; however, integral symplectic structures with the same
Chern invariant factors need not be cohomologous and thus may give rise to
inequivalent line bundles.
\end{remark}

\begin{nota}\label{not}
${\quad}$\newline

\noindent (i) Given a translation-invariant symplectic structure $\omega$ and a
translation-invariant Riemannian metric $h$ on $\rtm$, viewed as bilinear forms,
define a linear transformation $F:\rtm\to\rtm$ by the condition
$$\omega(u,v)=h(F(u),v)$$
for all $u,v\in\rtm$.  Let $\boldsymbol{h}$ and $\boldsymbol\omega$ denote the
Gram matrices of the
bilinear forms $h$ and $\omega$ with respect to the standard basis of $\rtm$.
The matrix of the
linear transformation $F$ in this basis is given by
$$\boldsymbol{F}=\boldsymbol{h}^{-1}\boldsymbol{\omega}.$$
Note that $F$ is antisymmetric relative to the inner product $h$,
and its eigenvalues are purely imaginary; we denote them by $\pm d_1^2 i,\dots,
\pm d_m^2 i$.

The linear transformation $F$ may be expressed in terms of the
``musical isomorphisms'': Given a finite-dimensional real vector space $V$ and a
nondegenerate bilinear form $B:V\times V\to\R$, denote by $B^{\flat}:V\to V^*$
the isomorphism from $V$ to its dual space given by $B^{\flat}(u)=B(\cdot,u)$,
i.e., $(B^{\flat}(u))(v)=B(v,u)$ for $u,v\in V$, and by $B^{\sharp}:V^*\to V$
the inverse of~$B^\flat$.    Then $F=h^\sharp\circ \omega^\flat$.

\smallskip
\noindent (ii) Let $M=\ztm\bs\rtm$.   Set
$$V_\omega=\sqrt{\det(\boldsymbol\omega)}
  \,\,=\int_M\,\frac{1}{m!}\omega^m \,,$$ the
\emph{symplectic volume} of $M$.
Since the standard basis of~$\rtm$ is a basis of~$\ztm$
we have, in particular, $V_{\omega_\rb}=r_1r_2\dots r_m$.

\end{nota}

\begin{prop}
We use Notation~\ref{notspec} and~\ref{not}. Let $M=\Z^{2m}\bs\rtm$, let
$\omega$ be a translation-invariant symplectic structure on $M$ representing an
integer cohomology class, and let $h$ be any flat metric on $M$. Given an
$m$-tuple $\j=(j_1,\dots, j_m)$ of nonnegative integers, let
$$
\nu(\j)=\pi\sum_{i=1}^m\,d_i^2(2j_i+1).
$$
Then $\Spec(k\omega,h)$ is the collection of all $\frac{1}{k}\nu(\j)$,
$\j\in\Z^m$, each counted $k^mV_\omega$ times.
\end{prop}

\begin{proof} Recalling Notation~\ref{notspec}, we see that
$2k^2\Spec(k\omega,h)$ is the spectrum of the operator in
Equation~(\ref{eq.lap}) acting on $L^2_k(P)$. Rather than carry out the
computation here, we refer to \cite{GW}, Section~3, where a similar computation
is performed. We indicate here how to translate the computation in \cite{GW} to
our setting.  We assume that $\omega=\omega_\rb$ for some $\rb=(r_1,\dots, r_m)$
as above.   Perform a change of coordinates on $\rtm$, letting $x_i'=r_ix_i$ and
$y_i'=y_i$ for $i=1,\dots, m$. In these new coordinates,
$\omega=\sum_{i=1}^m\,dx'_i\wedge dy'_i$, and the lattice~$\ztm$ is the
collection of all elements with coordinates in $r_1\Z\times\dots\times
r_m\Z\times\Z^m$. This change of coordinates aligns our notation with that in
\cite{GW}.   Next, in \cite{GW}, the operator under study is the Laplacian
$\Delta_P$ associated with the Riemannian metric on the Heisenberg manifold
$P=\Gamma\bs N$ induced by the left-invariant metric on $N$ for which the basis
$\{\tx_1,\dots,\tx_{2m},Z\}$ is orthonormal, where $\tx_1,\dots,\tx_{2m}$ are as
in Subsection~\ref{assochlb}.  We have $\Delta_P=\Delta +  (\rho_*Z)^2$ for
$\Delta$ as in~Equation~(\ref{eq.lap}).  (In the notation of \cite{GW}, we are
setting $g_{2m+1}$ equal to~$1$.) Writing $L^2(P)=\oplus_{k\in\Z}\,L^2_k(P)$,
then it is shown in \cite{GW} that for $k\neq 0$, the spectrum of $\Delta_P$
restricted to $L^2_k(P)$ is given by the collection of numbers
$4\pi^2k^2+2|k|\nu(\j)$, each occurring with multiplicity
$|k|^mr_1\dots r_m$.   (Our $|k|$ is denoted by $c$ in \cite{GW}.)  The operator
$(\rho_*Z)^2$ acts on $L^2_k(P)$ as multiplication by $4\pi^2k^2$.  Correcting
for this term and taking $k\in\Z^+$, we obtain the proposition.
\end{proof}

\begin{theorem}\label{main}
We use Notation~\ref{notspec} and~\ref{not}.  Let $\omega$ and $\omega'$ be two
translation-invariant symplectic structures on
$M$ representing integer cohomology classes, and let $h$ and $h'$ be flat
metrics on $M$.   Then the following are equivalent:
\begin{enumerate}
\item[(i)] $\Spec(\omega,h)=\Spec(\omega',h').$
\item[(ii)] $\Spec(k\omega, h)=\Spec(k\omega', h')$ for all $k\in\Z^+$.
\item[(iii)] The linear transformations $h^\sharp\circ \omega^\flat$ and
$h'^\sharp\circ \omega'^\flat$ (equivalently the matrices
$\boldsymbol{h}^{-1}\boldsymbol{\omega}$ and
$\boldsymbol{h'}^{-1}\boldsymbol{\omega'}$)
have the same eigenvalue spectrum, and $V_\omega=V_{\omega'}$.
\end{enumerate}
\end{theorem}

\begin{proof}
It is clear from Proposition~\ref{comp} that (i) and (ii) are equivalent and
that (iii) implies (i) and (ii). To see that (i) implies (iii), note that the
lowest eigenvalue occurring in $\Spec(\omega,h)$ is $\pi(d_1^2+\dots +d_m^2)$,
with multiplicity precisely $V_\omega$.  Thus $V_\omega$ is spectrally
determined. If we order the $d_j$ so that $d_1^2\leq d_2^2+\dots\leq d_m^2$,
then $2\pi d_1^2$ is the difference between the first two distinct eigenvalues
$\mu_1$ and $\mu_2$.   From the
multiplicity of $\mu_2$, we can determine how many of the $d_j^2$ equal $d_1^2$;
denote this number by $p$.  Since we know $\sum_{j=1}^m\,d_j^2$ from $\mu_1$ and
we know $d_1^2$, we can determine all eigenvalues $\nu(\j)$ for which
$j_{p+1}=\dots =j_m=0$, along with their multiplicities. Removing all these from
the spectrum, the lowest remaining eigenvalue is $\nu(\j)$ where $j_{p+1}=1$ and
all other $j_l$'s are zero. This enables us to determine $d_{p+1}^2$ and its
multiplicity, and we continue inductively.
\end{proof}

\begin{remark}\label{dets}
If the symplectic volumes $V_{\omega}$ and $V_{\omega'}$ coincide, then the
first part of condition (iii) in the previous theorem
can be replaced by the condition that $\boldsymbol{h}$ and $\boldsymbol{h'}$
have the same determinant, or equivalently that
$\vol(M,h)=\vol(M,h')$, since the determinant is multiplicative and
$V_\omega=\sqrt{\det\boldsymbol{\omega}}$.
\end{remark}

\begin{cor}\label{contin}
Let $\omega$ be a translation-invariant symplectic
structure on $M=\ztm\bs\rtm$ that represents an integer cohomology class. Given
any translation-invariant metric $h$ on $M$, let $\{h_t\}_t$ be a family of
metrics constructed as in the proof of Corollary~\ref{corclass}.  Then
$(M,h_t,\omega)$ is quantum equivalent as well as classically equivalent to
$(M,h,\omega)$ for all $t$.
\end{cor}

\begin{proof}
Classical equivalence was shown in Corollary~\ref{corclass}.
Quantum equivalence follows from Theorem~\ref{main}; in fact,
if $A_t$ (and hence $A_t^{-1}$) preserves~$\omega$ and
$h_t=A_t^*h$ then we have $$\boldsymbol{h}_t^{-1}\boldsymbol{\omega}=
\boldsymbol{A}_t^{-1}\boldsymbol{h}^{-1}\boldsymbol{\omega}
\boldsymbol{A}_t.$$
\end{proof}

\begin{defn}\label{d} We will say that $(M,h,\omega)$ is \emph{K\"ahler}, or
that the pair $(h,\omega)$ is a \emph{K\"ahler structure} on $M$, if there
exists a complex structure $J$ on $M$ such that $(M,h,J)$ is a K\"ahler manifold
whose associated K\"ahler form is $\omega$.
\end{defn}

\begin{prop}\label{ka}
The tuple $(M, h, \omega)$ is K\"ahler if and only if all the eigenvalues of
$h^\sharp\circ \omega^\flat$ (equivalently, of the matrix
$\boldsymbol{h}^{-1}\boldsymbol{\omega}$) are~$\pm i$.
\end{prop}

\begin{proof}
The latter condition is equivalent to $F^2=-\Id$ for the $h$-antisymmetric map
$F=h^\sharp\circ \omega^\flat$ from Notation~\ref{not}(i). But this is
equivalent to $(M, h, \omega)$ being
K\"ahler (with complex structure~$F$).
\end{proof}

\begin{cor}\label{cor.ka}
$\Spec(\omega, h)$ determines whether $(M,h,\omega)$ is K\"ahler. Moreover, any
two K\"ahler structures $(h,\omega)$ and $(h',\omega')$ that have the same
volume are quantum equivalent.
\end{cor}

Note that in the K\"ahler case, the symplectic volume~$V_\omega$
equals the Riemannian volume of $(M,h)$; recall Remark~\ref{dets}
together with Proposition~\ref{ka}.

For the construction of examples, we will restrict attention to metrics of the
form
$$
h_{\a,\b} =\sum_{j=1}^m\,(a_j^2\,dx_j^2 + b_j^2\,dy_j^2),
$$
when $\omega=\omega_\rb=\sum_{j=1}^m \,r_j\,dx_j\wedge dy_j$, where
$\a=(a_1,\dotsc,a_m), \b=(b_1,\dotsc,b_m)\in\R^m$.

\begin{remark}\label{rem.eig}
The eigenvalues of $\boldsymbol{h}_{\a,\b}^{-1}\,\boldsymbol{\omega}_\rb$ are
given by $\pm i\frac{r_1}{a_1b_1},\dots, \pm i\frac{r_m}{a_mb_m}$.
The symplectic volume
$V_{\omega_\rb}$ equals $r_1r_2\dots r_m$, by Notation $\ref{not}$(ii).
\end{remark}

\begin{examples}\label{ex.ka}
$\text{ }$

\noindent (i) Let $m=2$. Set $h=dx_1^2+dy_1^2+dx_2^2+4\,dy_2^2$, so the
representing matrix is
$$\boldsymbol{h}=\begin{bmatrix}1&0&0&0\\0&1&0&0\\0&0&1&0\\0&0&0
&4\end{bmatrix},$$
and let
$$\omega=2\,dx_1\wedge dy_1 + 2\,dx_2\wedge dy_2\,\,\,\,\text{  and
}\,\,\,\,\omega'=dx_1\wedge dy_1 +4\,dx_2\wedge dy_2.$$

Then both $\boldsymbol{h}^{-1}\boldsymbol{\omega}$ and $\boldsymbol{h}^{-1}
\boldsymbol{\omega'}$ have
eigenvalues $\pm i$ and $\pm2i$.  Thus $\omega$ and $\omega'$ are quantum
equivalent magnetic fields on $(\Z^4\bs\R^4,h)$.  The two structures
$(\Z^4\bs\R^4, h, \omega)$ and $(\Z^4\bs\R^4,h,\omega')$ are not K\"ahler.

\noindent (ii)  Set $h=dx_1^2+4\,dy_1^2+dx_2^2+4\,dy_2^2$, so the representing
matrix is
$$\boldsymbol{h}=\begin{bmatrix}1&0&0&0\\0&4&0&0\\0&0&1&0\\0&0&0&4
\end{bmatrix},$$
$$\omega=2\,dx_1\wedge dy_1 + 2\,dx_2\wedge dy_2,$$
$h'=dx_1^2+dy_1^2+4\,dx_2^2+4\,dy_2^2$, so the representing matrix is
$$\boldsymbol{h'}=
\begin{bmatrix}1&0&0&0\\0&1&0&0\\0&0&4&0\\0&0&0&4\end{bmatrix},$$
and
$$\omega'=dx_1\wedge dy_1 +4\,dx_2\wedge dy_2.
$$
Then all eigenvalues of $\boldsymbol{h}^{-1}\boldsymbol{\omega}$ and of
$\boldsymbol{h'}^{-1}\boldsymbol{\omega'}$ are $\pm i$.
Thus $(\Z^4\bs\R^4,h,\omega)$ and
$(\Z^4\bs\R^4,h',\omega')$ are quantum equivalent K\"ahler structures.
Note that $h$ and $h'$ are isometric via the map that interchanges the
coordinates $y_1$ and $x_2$; so $\omega$ and the corresponding pullback of
$\omega'$ can be viewed as quantum equivalent K\"ahler structures on the same
underlying Riemannian manifold.
\end{examples}

\begin{remark} The first of the two examples above first appeared in a slightly
different context in \cite{GGKW}.
\end{remark}

In examples of pairs of quantum equivalent line bundles
arising from Theorem~\ref{main}, the cotangent bundles
endowed with the associated symplectic forms will in general
be nonsymplectomorphic. In particular, this is the case for
the pairs in Example~\ref{ex.ka}. In fact, we have:

\begin{prop}\label{noneq}
Let $\omega, \omega'$ be two translation-invariant symplectic
structures on the torus $M=\ztm\bs\rtm$ with Chern invariant factors
$\rb=(r_1,\ldots,r_m)$ and $\rb'=(r'_1,\ldots,r'_m)$, respectively.
Let $\Omega:=\pi^*\omega+\omega_0$
and $\Omega:=\pi^*\omega'+\omega_0$ be the associated symplectic
forms on $T^*M$, where $\omega_0$ is the Liouville form.
If $\rb\ne\rb'$, then $(T^*M,\Omega)$ and $(T^*M,\Omega')$
are not symplectomorphic.
\end{prop}

\begin{proof}
For $k=1,\ldots,m$, we consider the values of the integer cohomology classes
of $T^*M$ represented
by $\Omega^k:=\Omega\wedge\ldots\wedge\Omega$ on integer homology classes
of~$T^*M$.
We have $T^*M\cong M\times\rtm$. In particular, each integer homology
class in $H_{2k}(T^*M;\Z)$ can be represented by a suitable smooth closed cycle
in $M\times\{0\}$
(a finite sum of oriented $2k$-dimensional subtori).
We consider the integrals of $\Omega^k:=\Omega\wedge
\ldots\wedge\Omega$ over such $2k$-cycles. These are equal to the integrals
of $\omega^k$ over the corresponding cycles in~$M$.
Obviously, the minimal nonzero absolute value of these integrals is
$r_1\cdot\ldots\cdot r_k$. Thus, if there were a symplectomorphism
$(T^*M,\Omega)\to(T^*M,\Omega')$, then $r_1\cdot\ldots\cdot r_k=
r'_1\cdot\ldots\cdot r'_k$ for each $k=1,\ldots,m$, and thus $\rb=\rb'$.
\end{proof}

The next proposition shows that the previous result continues to hold if we
remove the zero
section from the cotangent bundle, as it might seem natural to do
in some contexts (see the comments after Theorem \ref{2}): Let $T^*M
\smallsetminus0\cong
M\times(\rtm \smallsetminus\{0\})$ denote the manifold of all nonvanishing
cotangent vectors to~$M$ (this is an open submanifold of $T^*M$).

\begin{prop}\label{noneq0}
In the situation of Proposition~\ref{noneq}, $\rb\ne\rb'$
also implies that $(T^*M \smallsetminus0,\Omega)$ and $(T^*M
\smallsetminus0,\Omega')$
are not symplectomorphic.
\end{prop}

\begin{proof}
Let $X\in\rtm \smallsetminus\{0\}$ be arbitrary. For $j\le 2m-2$,
the $j$th homology group of $T^*M \smallsetminus0\cong M\times(\rtm
\smallsetminus\{0\})$
is still isomorphic to the $j$th homology group of $M$, and each of
its cycles can be represented by a suitable cycle in $M\times\{X\}$.
Therefore, by the same argument as in the proof of Proposition~\ref{noneq}
we see that the symplectomorphism class of~$\Omega$ determines
$r_1,\ldots,r_{m-1}$. In order to see that it also determines~$r_m$,
note that $H_{2m}(M\times(\rtm \smallsetminus\{0\});\Z)=\Z\oplus\Z^{2m}$,
where $\Z$ corresponds to $H_{2m}(M;\Z)$ and
$\Z^{2m}$ is generated by products of $1$-cycles in~$M$
with a $(2m-1)$-cycle in $\rtm \smallsetminus\{0\}$
generating $H_{2m-1}(\rtm \smallsetminus\{0\};\Z)$.
Since the integral of $\Omega^m$ over such products vanishes,
we still have that the minimal nonzero absolute value of the integrals
of~$\Omega^m$
over $2m$-cycles representing integral homology classes in $T^*M
\smallsetminus0$
is $r_1\cdot\ldots\cdot r_m$.
\end{proof}


\begin{thebibliography}{99}


\bibitem{CK} C. Croke and B. Kleiner. Conjugacy and rigidity for manifolds
with a parallel vector field. \textit{J. Differential Geom.} \textbf{39} (1994),
no.~3,  659--680.

\bibitem{DG} J. Duistermaat and V. Guillemin.  The spectrum of positive elliptic
operators and periodic bicharacteristics.  \textit{Invent. Math} \textbf{29}
(1975), 39--79.

\bibitem{GGKW} C. Gordon, P. Guerini, T. Kappeler, and D. Webb. Inverse spectral
results on even dimensional tori. \textit{Ann. Inst. Fourier} \textbf{58}
(2008), 2245--2501.

\bibitem{GKSW} C. Gordon, W. Kirwin, D. Schueth, and D. Webb.
Quantum equivalent magnetic fields that are not classically
equivalent. \textit{Ann. Inst. Fourier} (to appear).

\bibitem{GW} C. Gordon and E. Wilson. The spectrum of the Laplacian
on Riemannian Heisenberg manifolds. \textit{Michigan Math. J.} \textbf{33}
(1986),
no.~2, 253--271.

\bibitem{GrHa} P. Griffiths and J. Harris. \textit{Principles of Algebraic
Geometry}.  John Wiley \& Sons, 1978.

\bibitem{HdL} G. Heckman and T. de Laat. On the regularization of the Kepler
problem.
\textit{preprint:} \texttt{arXiv:1007.3695}

\bibitem{LS} T. Ligon and M. Schaaf: On the global symmetry of the
classical Kepler problem, \textit{Reports on Math. Phys.} \textbf{9} (1976),
281--300.

\bibitem{Woo} N. M. J. Woodhouse.
\textit{Geometric Quantization}, second Edition. Oxford
University Press, Inc., New York, 1991.

\end{thebibliography}
\end{document}